\documentclass[11pt]{amsart}

\title[Finite presentations of the mapping class groups of Heegaard splittings ]
{Finite presentations of the mapping class groups of once-stabilized Heegaard splittings}

\author{Daiki Iguchi}
\address{
Department of Mathematics \newline
\indent Hiroshima University, 1-3-1 Kagamiyama, Higashi-Hiroshima, 739-8526, Japan}
\email{diguchi00@gmail.com}

\date{\today}

\usepackage{amsthm}
\usepackage{mathrsfs}
\usepackage{latexsym}
\usepackage[dvips]{graphicx}
\usepackage[dvips]{psfrag}
\usepackage[dvips]{color}
\usepackage{xypic}
\usepackage[abs]{overpic}
\usepackage{caption}
\usepackage[all]{xy}
\usepackage[normalem]{ulem}
\usepackage{marginnote}
\usepackage{comment}
\usepackage{cite}
\usepackage {overpic}

\theoremstyle{plain}
\newtheorem*{theorem*}{Theorem}
\newtheorem*{lemma*} {Lemma}
\newtheorem*{corollary*} {Corollary}
\newtheorem*{proposition*}{Proposition}
\newtheorem*{conjecture*}{Conjecture}
\newtheorem*{claim*}{Claim}
\newtheorem{theorem}{Theorem}[section]
\newtheorem{lemma}[theorem]{Lemma}
\newtheorem{corollary}[theorem]{Corollary}
\newtheorem{proposition}[theorem]{Proposition}

\newtheorem{claim}{Claim}

\theoremstyle{remark}

\newtheorem*{definition}{Definition}

\newtheorem{remark}{Remark}

\newtheorem*{ack}{Acknowledgements}

\theoremstyle{definition}

\newtheoremstyle{citing}
  {}
  {}
  {\itshape}
  {}
  {\bfseries}
  {.}
  {.5em}
  {\thmnote{#3}}

\theoremstyle{citing}

\newcommand{\MCG}{\mathrm{MCG}}
\newcommand{\Diff}{\mathrm{Diff}}

\newcommand{\id}{\mathrm{id}}

\newcommand{\Cl}{\operatorname{Cl}}
\newcommand{\Int}{\operatorname{Int}}

\newcommand{\Isot}{\mathrm{Isot}}
\newcommand{\He}{\mathcal{H}}
\newcommand{\mathC}{\mathcal{C}}

\newcommand{\Unk}{\mathrm{Unk}}
\newcommand{\Rsc}{\mathscr{R}}
\newcommand{\Sig}{\Sigma}
\newcommand{\vph}{\varphi}
\newcommand{\om}{\omega}
\newcommand{\lam}{\lambda}
\newcommand{\ep}{\epsilon}
\newcommand{\Af}{\mathfrak{A}}

\newcommand{\Ff}{\mathfrak{F}}
\newcommand{\pp}{\prime+}
\newcommand{\pn}{\prime-}

\newcommand{\wtil}{\widetilde}



\begin{document}

\maketitle

\begin{abstract}  
Let $g \ge 2$ and 
assume that we are given a genus $g$ Heegaard splitting 
of a closed orientable $3$-manifold 
with the distance greater than $2g+2$. 
We prove that the mapping class group of the once-stabilization 
of such a Heegaard splitting is finitely presented. 
\end{abstract}

\vspace{1em}

\renewcommand{\thefootnote}{}
\footnote{\textbf{2020 Mathematics Subject Classification}: 57K30, 57M60} 
\footnote{\textbf{Keywords}: 
3-manifold, mapping class group, Heegaard splitting, space of Heegaard splittings.}

\section{Introduction}\label{sec:introduction}

Let $(M,\Sig)$ be a Heegaard splitting 
of a compact orientable $3$-manifold $M$. 
The {\em mapping class group} $\MCG(M,\Sig)$ of the Heegaard splitting $(M,\Sig)$ 
is defined to be the group $\pi_{0}(\Diff(M,\Sig))$ of path-connected components  
of the group $\Diff(M,\Sig)$, 
where we denote by $\Diff(M,\Sig)$ the group of diffeomorphisms of $M$ 
that preserve $\Sig$ setwise.  
There is a natural homomorphism from $\MCG(M,\Sig)$ 
to the mapping class group $\MCG(M)$ of $M$. 
Following Johnson \cite{Jo13}, we call the kernel of this natural homomorphism 
the {\em isotopy subgroup} of $\MCG(M,\Sig)$, and denote it by $\Isot(M,\Sig)$. 

In this paper, we are interested in 
the isotopy subgroup of the mapping class group of  a once-stabilized Heegaard splitting. 
Let $(M,\Sig')$ be a genus $g(\Sig') \ge 2$ Heegaard splitting 
of a closed orientable $3$-manifold $M$. 
We say that a Heegaard splitting $(M,\Sig)$ 
is a (once-){\em stabilization} of $(M,\Sig')$ if 
it is obtained from $(M,\Sig')$ by adding a $1$-handle whose core is parallel into $\Sig'$.  
Corresponding to two handlebodies $V^-_{\Sig'}$ and $V^+_{\Sig'}$  in $M$ 
with $\partial V^-_{\Sig'}=\partial V^+_{\Sig'}=\Sig'$, 
there are two obvious subgroups of $\Isot(M,\Sig)$:  
one is $\Isot(V^-_{\Sig'},\Sig^-)$ and the other is $\Isot(V^+_{\Sig'},\Sig^+)$, 
where $\Sig^-$ (resp. $\Sig^+$) is the Heegaard surface obtained by pushing $\Sig$ 
into $V^-_{\Sig'}$ (resp. $V^+_{\Sig'}$) slightly. 
It is natural to ask when these subgroups generate $\Isot(M,\Sig)$. 
In  \cite{Jo13}, Johnson proved that 
if the distance $d(\Sig')$ of the Heegaard splitting $(M,\Sig')$ 
is greater than $2g(\Sig')+2$, then 
the two subgroups defined above generate $\Isot(M,\Sig)$. 
As a consequence of this fact together with a result in Scharlemann \cite{Sc13} 
that says $\Isot(V^\pm_{\Sig'},\Sig^\pm)$ are finitely generated,  
it follows that $\Isot(M,\Sig)$ and $\MCG(M,\Sig)$ are finitely generated. 
In that paper, Johnson conjectured 
that $\Isot(M,\Sig)$ is an amalgamation of the two groups 
$\Isot(V^-_{\Sig'},\Sig^-)$ and $\Isot(V^+_{\Sig'},\Sig^+)$. 
This is the main result of the paper: 

\begin{theorem}\label{thm:amalgamation}
Suppose that $(M,\Sig')$ is Heegaard splitting of 
a closed orientable $3$-manifold $M$ with  
$d(\Sig') >2g(\Sig')+2$, and that 
$(M,\Sig)$ is a once-stabilization of $(M,\Sig')$. 
Suppose that $(V^-_{\Sig'},\Sig^-)$ (resp. $(V^+_{\Sig'},\Sig^+)$) is the Heegaard splitting 
of $V^-_{\Sig'}$ (resp. $V^+_{\Sig'}$) 
obtained by pushing $\Sig$ 
into $V^-_{\Sig'}$ (resp. $V^+_{\Sig'}$) slightly, 
where $V^-_{\Sig'}$ and $V^+_{\Sig'}$ are handlebodies in $M$ bounded by $\Sig'$. 
Then, $\Isot(M,\Sig)$ is isomorphic 
to an amalgamation of the two groups $\Isot(V^-_{\Sig'},\Sig^-)$ 
and $\Isot(V^+_{\Sig'},\Sig^+)$.  
\end{theorem}

One might expect 
that the above theorem has something to do with van Kampen's theorem. 
This idea can be justified as follows. 
Following Johnson-McCullough \cite{JM13}, 
we define the space $\He(M,\Sig)$ to be $\Diff(M)/\Diff(M,\Sig)$ and 
call it the {\em space of Heegaard splittings} equivalent to $(M,\Sig)$. 
Let $\He$ denote the path-connected component of $\He(M,\Sig)$ 
containing the left coset $\mathrm{id}_M \cdot \Diff(M,\Sig)$. 
It is known that 
if a $3$-manifold admits a Heegaard splitting with the distance greater than two, 
then such a $3$-manifold must be hyperbolic. 
By a result in \cite{JM13} (see Theorem~\ref{thm:JM13} below for more details) 
together with this fact, 
it follows that $\Isot(M,\Sig)$ is isomorphic to $\pi_1(\He)$. 

Now fix a {\em spine} $K=K^- \cup K^+$ of the Heegaard splitting $(M,\Sig')$, that is, 
$K^-$ and $K^+$ are finite graphs embedded in $M$ such that 
the complement $M \setminus K$ is diffeomorphic to $\Sig' \times (-1,1)$ 
and $\Sig'$ is a slice of this product structure.  
Denote by $\He^-$ (resp. $\He^+$) the subspace of $\He$ 
consisting of those elements represented by a Heegaard surface $T$ such that 
$T$ is a genus $g(\Sig')+1$ Heegaard surface 
of the genus $g(\Sig')$ handlebody $M \setminus \Int(N(K^+))$ 
(resp. $M \setminus \Int(N(K^-))$), 
where $N(K^+)$ (resp. $N(K^-)$)  
is a small neighborhood of $K^+$ (resp. $K^-$). 
By the similar reason as above (see Theorem~\ref{thm:large_manifolds} below), 
we can identify $\Isot(V^-_{\Sig'},\Sig^-)$ and $\Isot(V^+_{\Sig'},\Sig^+)$ with 
the fundamental groups $\pi_1(\He^-)$ and $\pi_1(\He^+)$ respectively. 
Set $\He^\cup:=\He^- \cup \He^+$. 
Theorem~\ref{thm:amalgamation} is a corollary of the following.  
 
\begin{theorem}\label{thm:homotopy_equivalence}
The inclusion $\He^{\cup} \rightarrow \He$ is a homotopy equivalence.  
\end{theorem}

\noindent It is well known that 
a genus $g+1$ Heegaard splitting of a genus $g$ handlebody 
is unique up to isotopy. 
Similarly, a genus $g+1$ Heegaard splitting of the space $F_g \times [-1,1]$
is unique up to isotopy, 
where we denote by $F_g$ a closed genus $g$ surface. 
In other words, $\He^+$, $\He^-$ and $\He^- \cap \He^+$ are all connected,  
and hence van Kampen's Theorem applies to 
the triple $(\He^-,\He^+,\He^- \cap \He^+)$. 

The proof of Theorem~\ref{thm:homotopy_equivalence} is based on  
the concept of {\em graphics}, 
which was first introduced by Cerf \cite{Cer68} and then 
successfully applied to the study of Heegaard splittings by 
Rubinstein and Scharlemann \cite{RS96}. 
More precisely, we prove Theorem~\ref{thm:homotopy_equivalence} 
by generalizing the method developed by Johnson \cite{Jo10,Jo13}. 
We also use an argument due to Hatcher \cite{Hat76} crucially,  
which is a parametrized version of the innermost disk argument. 

In Section~\ref{sec:conclusion}, 
we confirm that the isotopy subgroup of a genus $g+1$ Heegaard splitting of 
a genus $g$ handlebody is finitely presented: 

\begin{theorem}\label{thm:handlebody}
Let $V$ be a handlebody of genus $g(V) \ge 2$, and 
let $(V,\Sig)$ be a genus $g(V)+1$ Heegaard splitting of $V$. 
Then, $\Isot(V,\Sig)$ is finitely presented. 
\end{theorem}

It follows from Theorem~\ref{thm:handlebody} that 
$\pi_1(\He^-)$ and $\pi_1(\He^+)$ are finitely presented. 
In consequence, we have 

\begin{corollary}\label{cor:conclusion}
Let $(M,\Sig')$ be a Heegaard splitting 
of a closed orientable $3$-manifold $M$ with  
$d(\Sig') >2g(\Sig')+2$. 
Let $(M,\Sig)$ be a once-stabilization of $(M,\Sig')$. 
Then, $\Isot(M,\Sig)$ and $\MCG(M,\Sig)$ are finitely presented. 
\end{corollary}

We remark that a problem related to this work was treated in Koda-Sakuma \cite{KS21}. 
In that paper, the concept of the ``homotopy motion group" was introduced, 
and they considered the question 
that asks when the homotopy motion group $\Pi(M,\Sig)$ of a Heegaard surface 
in a $3$-manifold $M$ can be written as an amalgamation of the two 
homotopy motion groups $\Pi(U^-_\Sig,\Sig)$ and $\Pi(U^+_\Sig,\Sig)$ corresponding 
to the two handlebodies $U^-_\Sig$ and $U^+_\Sig$ 
with $\partial U^-_\Sig=\partial U^+_\Sig=\Sig$.  

The paper is organized as follows. 
In Section~\ref{sec:preliminaries}, 
we recall from \cite{JM13} some facts about the space of Heegaard splittings. 
We also recall the definition of the distance of a Heegaard splitting. 
To prove Theorem~\ref{thm:homotopy_equivalence}, 
we will need  
to deal with the graphic determined by a $4$-parameter family of Heegaard surfaces. 
In Section~\ref{sec:sweep-out}, we give a quick review of the theory of graphics, 
and then we see that some ideas in \cite{Jo10} can be adapted to our setting. 
In Section~\ref{sec:poof}, 
we prove Theorem~\ref{thm:homotopy_equivalence}. 
Finally, we give the proof of 
Theorem~\ref{thm:handlebody} in Section~\ref{sec:conclusion}. 

\begin{ack}
The author would like to thank his advisor Yuya Koda 
for many advice and sharing his insight. 
I am also grateful to the anonymous referees for their valuable comments 
that improved the manuscript.  
This work was supported by JSPS KAKENHI Grant Number JP21J10249. 
\end{ack}

\section{Preliminaries}\label{sec:preliminaries}
Throughout the paper, we will use the following notations.  
For a topological space $X$, 
we denote by $|X|$ the number of path-connected components of $X$. 
For a subspace $Y$ of $X$,  
$\Int(Y)$ and $\Cl(Y)$ denote 
the interior and the closure of $Y$ in $X$, respectively. 
We will denote by $J$ the closed interval $[-1,1]$.  

\subsection{The space of Heegaard splittings}\label{sub:space_of H-split}

Let $M$ be a compact orientable $3$-manifold (possibly with boundary). 
Let $(M,\Sig)$ be a Heegaard splitting of  $M$. 
This means that $\Sig \subset M$ is a closed orientable embedded surface 
cutting $M$ into the two compression bodies.  
Here, a {\em compression body} is a $3$-manifold with nonempty boundary 
admitting a Morse function without critical points of index $2$ and $3$. 
A handlebody is a typical example of a compression body.  
The space $\He(M,\Sig)=\Diff(M)/\Diff(M,\Sig)$ is called 
the {\em space of Heegaard splittings} equivalent to $(M,\Sig)$. 
Note that there is a one-to-one correspondence between 
$\He(M,\Sig)$ and 
the set of images of $\Sig$ under diffeomorphisms of $M$. 
We often identify an element of  $\He(M,\Sig)$ 
with the corresponding Heegaard surface. 
We always take the surface $\Sig$ as the basepoint of $\He(M,\Sig)$, 
which corresponds to the left coset $\id_M \cdot \Diff(M,\Sig)$. 
The space $\He(M,\Sig)$ admits a structure of a Fr\'{e}chet manifold, and 
this implies that $\He(M,\Sig)$ has the homotopy type of a CW complex. 

\begin{theorem}[Johnson-McCullough {\cite[Corollary~1]{JM13}}]
\label{thm:JM13}
Suppose that $M$ is closed, orientable, irreducible 
and $\pi_1(M)$ is infinite, 
and that $M$ is not a non-Haken infranilmanifold. 
Then $\pi_{k}(\He(M,\Sig))=0$ for $k  \ge 2$, 
and there is an exact sequence
\[1 \rightarrow Z(\pi_{1}(M)) \rightarrow 
\pi_{1}(\He(M,\Sig)) \rightarrow \Isot(M,\Sig) \rightarrow 1.\]
\end{theorem}

A similar thing as above holds for handlebodies and the space $F_g \times J$: 

\begin{theorem}\label{thm:large_manifolds}
Let $g' \ge g \ge 2$. 
Suppose that $M$ is a genus $g$ handlebody or the space $F_g \times J$, 
where $F_g$ denotes a closed orientable surface of genus $g$. 
Suppose that $(M,\Sig)$ is a genus $g'$ Heegaard splitting of $M$. 
Then, we have 
$\pi_{1}(\He(M,\Sig)) \cong \Isot(M,\Sig)$ and 
$\pi_{k}(\He(M,\Sig))=0$ for $k  \ge 2$. 
\end{theorem}

\begin{proof}
By Theorem~1 in \cite{JM13}, $\pi_k(\He(M,\Sig))=\pi_k(\Diff(M))$ for $k \ge 2$, 
and there is an exact sequence 
\[1 \rightarrow \pi_{1}(\Diff(M)) \rightarrow 
\pi_{1}(\He(M,\Sig)) \rightarrow \Isot(M,\Sig) \rightarrow 1.\]
By Earle-Eells \cite{EE69} and Hatcher \cite{Hat76}, 
we have 
$\pi_k(\Diff(M))=0$ for $k \ge 1$. 
\end{proof}

\subsection{The distance of a Heegaard splitting}\label{sub:distance}
Let $(M,\Sig')$ be a genus $g(\Sig') \ge 2$ Heegaard splitting 
of a closed orientable $3$-manifold $M$.  
Denote by $V^-_{\Sig'}$ and $V^+_{\Sig'}$ the handlebodies in $M$ with 
$V^-_{\Sig'} \cap V^+_{\Sig'}=\partial V^-_{\Sig'}=\partial V^+_{\Sig'}=\Sig'$.  
The {\em curve graph} $\mathcal{C}(\Sig')$ is the graph defined as follows.  
The vertices of $\mathcal{C}(\Sig')$ are isotopy classes 
of nontrivial simple closed curves in $\Sig'$, and 
the edges are pairs of vertices that admit disjoint representatives. 
We denote by $d_{\mathcal{C}(\Sig')}$ the simplicial metric on $\mathcal{C}(\Sig')$. 

Let $\mathcal{D}^-$ (resp. $\mathcal{D}^+$) denote 
the set of vertices in $\mathcal{C}(\Sig')$ that 
are represented by simple closed curves bounding disks in 
$V^-_{\Sig'}$ (resp. $V^+_{\Sig'}$).   
Then, the (Hempel) {\em distance} $d(\Sig')$ of the Heegaard splitting $(M,\Sig')$ 
is defined to be 
\[d(\Sig'):=d_{\mathcal{C}(\Sig')}(\mathcal{D}^-,\mathcal{D}^+).\] 

For example, if $M$ contains an essential sphere, 
then any Heegaard splitting of $M$ has the distance zero (cf. Haken \cite{Hak68}). 
If $M$ contains an essential torus, 
then any Heegaard splitting of $M$ has the distance at most two. 
Furthermore, any Heegaard splitting of a Seifert manifold has the distance 
at most two. 
See Hempel \cite{He01} for these two facts. 
As a consequence of the Geometrization Theorem and these facts, 
we have  

\begin{theorem}\label{thm:dist-hyp}
Suppose that $(M,\Sig')$ is a Heegaard splitting of a closed orientable $3$-manifold $M$. 
If $d(\Sig') >2$, then $M$ admits a hyperbolic structure. 
\end{theorem}

\section{Sweep-outs and graphics}\label{sec:sweep-out} 
In this section, 
we recall the definition of graphics and summarize their properties.  
In what follows, let $M$ denote a closed orientable $3$-manifold. 

\subsection{Graphics}\label{sub:graphics} 
Let $(M,\Sig)$ be a Heegaard splitting of $M$. 
A {\em sweep-out} associated with $(M,\Sig)$ is a function $h:M \rightarrow J=[-1,1]$ 
such that the level set $h^{-1}(t)$ is a Heegaard surface 
isotopic to $\Sig$ if $t \in \Int(J)$, 
and $h^{-1}(t)$ is a finite graph in $M$ if $t \in \partial  J$. 
The preimage $h^{-1}(\partial J)$ is called the {\em spine} of $h$.  

\begin{lemma}\label{lem:family_sweepout}
Let $n >0$ and 
$(M,\Sig)$  a Heegaard splitting of a closed orientable $3$-manifold $M$. 
Let $\vph: D^n \rightarrow \He(M,\Sig)$. 
Then, there exists a family $\{h_u:M \rightarrow J \mid u \in D^n\}$ 
of sweep-outs such that 
$h_u^{-1}(0)=\vph(u)$ for $u \in D^n$. 
\end{lemma}

\begin{proof}
Take a sweep-out $h:M \rightarrow J$ with $h^{-1}(0)=\Sig$. 
We note that $\Diff(M) \rightarrow \Diff(M)/\Diff(M,\Sig)=\He(M,\Sig)$ is a fibration 
\cite{JM13}. 
So, the map $\vph$ lifts to a map $\widetilde{\vph}:D^n \rightarrow \Diff(M)$. 
Now define $h_u:=h \circ \widetilde{\vph}(u)^{-1}$ for $u \in D^n$. 
\end{proof}

Let $(M,\Sig)$ and $(M,\Sig')$ be Heegaard splittings of $M$. 
Let $f:M \rightarrow J$ be a sweep-out with $f^{-1}(0)=\Sig'$.  
Furthermore, let $\{h_u:M \rightarrow J \mid u \in D^2\}$ 
be a family of sweep-outs associated with $(M,\Sig)$. 
We define the map $\Phi:M \times D^2 \rightarrow J^2 \times D^2$ by 
\[\Phi(x,u)=(f(x),h_u(x),u).\] 
Set $L:=\Phi^{-1}(\partial J^2 \times D^2)$, and 
$W:=(M \times D^2) \setminus L$. 
Define $S=S(\Phi|_W)$ to be the set of all points $w \in W$ 
such that $\mathrm{rank}\,d(\Phi|_W)_w <4$. 
The image $\Gamma$ of $S$ in $J^2 \times D^2$ 
is called the {\em graphic} defined by $f$ and $\{h_u\}$.  

After a small perturbation, 
we may assume that the map $\Phi$ is generic in the following sense. 
First, for $u \in D^2$, the spine $h_u^{-1}(\partial J)$ intersects each level set of $f$ 
at finitely many points.  
Similarly, for $u \in D^2$, 
the spine $f^{-1}(\partial J)$ intersects each level set of $h_u$ at finitely many points. 
Furthermore, $\Phi$ is ``excellent" on $W$.  
This means that the set $S$ of singular points of $\Phi|_W$ 
is a $3$-dimensional submanifold in $W$, and  
$S$ is divided into four parts $S_2$, $S_3$, $S_4$ and $S_5$,  
where $S_k$ consists of singular points of codimension $k$. 
(In the notation of \cite{Bo67}, 
we can write $S_2=\Sig^{2,0}$, $S_3=\Sig^{2,1,0}$, $S_4=\Sig^{2,1,1,0}$ and 
$S_5=\Sig^{2,1,1,1,0} \cup \Sig^{2,2,0}$.) 
\renewcommand{\thefootnote}{\arabic{footnote}}
For $k \neq 5$, 
$\Phi$ has one of the following canonical forms around a point $w \in S_k$: 
\footnote[1]{We do not know if there exist canonical forms  
for the singularities of type $\Sig^{2,2,0}$. 
However, the singularities in $S_5$ 
are not  important for our present purpose.}
there exist local coordinates $(a,b,c,x,y)$ centered at $w$ and 
$(A,B,X,Y)$ centered at $\Phi(w)$ such that 
\begin{equation*}
(A \circ \Phi, B \circ \Phi, X \circ \Phi, Y \circ \Phi)=
\begin{cases}
(a,b,c,x^2+y^2) & \text{definite fold ($w \in S_2$)} \\
(a,b,c,x^2-y^2) & \text{indefinite fold ($w \in S_2$)} \\
(a,b,c,x^3+ax-y^2) & \text{cusp ($w \in S_3$)}\\
(a,b,c,x^4+ax^2+bx+y^2) & \text{definite swallowtail ($w \in S_4$)}\\
(a,b,c,x^4+ax^2+bx-y^2) & \text{indefinite swallowtail ($w \in S_4$)}
\end{cases}
\end{equation*}
Furthermore, for $2 \le k \le 5$, $\Phi|_{S_k}$ is an immersion with normal crossings, 
and the images of the $S_k$'s are in general position. 
The main reference about these materials is the book \cite{GG73} 
by Golubitsky and Guillemin. 
Hatcher-Wagoner \cite{HW73} also contains a helpful review for our present purpose. 

In the remaining part of the paper, 
we always assume that the map $\Phi$ 
has the property described above. 
Under this assumption, 
$\Gamma$ has the natural stratification: 
we can write $\Gamma=F_3 \cup F_2 \cup F_1 \cup F_0$, where 
$\mathrm{dim}\,F_k=k$ for $0 \le k \le 3$ 
and each $F_k$ has the following description.  
\begin{enumerate}
\item[$F_3$:] 
This consists of those points $y \in \Gamma$ such that 
$(\Phi|_S)^{-1}(y) \subset S_2$ and $|(\Phi|_S)^{-1}(y)|=1$. 
\item[$F_2$:] 
This consists of those points $y \in \Gamma$ such that 
\begin{itemize}
\item $(\Phi|_S)^{-1}(y) \subset S_2$ and $|(\Phi|_S)^{-1}(y)|=2$, or 
\item $(\Phi|_S)^{-1}(y) \subset S_3$ and $|(\Phi|_S)^{-1}(y)|=1$. 
\end{itemize}
\item[$F_1$:] 
This consists of those points $y \in \Gamma$ such that 
\begin{itemize}
\item $(\Phi|_S)^{-1}(y) \subset S_2$ and $|(\Phi|_S)^{-1}(y)|=3$, 
\item $(\Phi|_S)^{-1}(y) \subset S_2 \cup S_3$ and $|(\Phi|_S)^{-1}(y)|=2$, or
\item $(\Phi|_S)^{-1}(y) \subset S_4$ and $|(\Phi|_S)^{-1}(y)|=1$.
\end{itemize}
\item[$F_0$:]
This consists of those points $y \in \Gamma$ such that 
\begin{itemize}
\item $(\Phi|_S)^{-1}(y) \subset S_2$ and $|(\Phi|_S)^{-1}(y)|=4$, 
\item $(\Phi|_S)^{-1}(y) \subset S_2 \cup S_3$ and $|(\Phi|_S)^{-1}(y)|=3$, 
\item $(\Phi|_S)^{-1}(y) \subset S_2 \cup S_4$ and $|(\Phi|_S)^{-1}(y)|=2$, 
\item $(\Phi|_S)^{-1}(y) \subset S_3$ and $|(\Phi|_S)^{-1}(y)|=2$, or
\item $(\Phi|_S)^{-1}(y) \subset S_5$ and $|(\Phi|_S)^{-1}(y)|=1$. 
\end{itemize}
\end{enumerate}

\subsection{Labeling the regions of $J^2 \times D^2$}\label{sub:labeling}
In this subsection, 
we will see that some definitions in \cite{Jo10} 
can be modified slightly and adapted to our setting. 

Let $(M,\Sig)$ be a Heegaard splitting. 
We assume that one component of $M \setminus \Sig$ is assigned  
the label $-$ and the other is assigned the label $+$ in some way. 
We denote by $U^-_\Sig$ and $U^+_\Sig$ the components of $M \setminus \Sig$ 
labeled by $-$ and $+$ respectively. 
(Typically, such a labeling is determined by a given sweep-out $h$ with $h^{-1}(0)=\Sig$.   
In this case, we can define $U^-_\Sig=h^{-1}([-1,0])$ and $U^+_\Sig=h^{-1}([0,1])$.) 
Such an assignment of the labels $-$ or $+$ to the components of $M \setminus \Sig$ is 
called a {\em transverse orientation} of $\Sig$.  

\begin{definition}\label{def:above_below}
Let $(M,\Sig)$ and $U^{\pm}_\Sig$ be as above. 
Suppose $\Sig' \subset M$ is a closed embedded surface. 
Then, we say that $\Sig'$ is {\em mostly above} $\Sig$ 
if $\Sig'$ is transverse to $\Sig$, and 
if every component of $\Sig' \cap U^-_\Sig$ 
is contained in a disk subset of $\Sig'$. 
Similarly, we say that $\Sig'$ is {\em mostly below} $\Sig$ 
if $\Sig'$ is transverse to $\Sig$, and 
if every component of $\Sig' \cap U^+_\Sig$ 
is contained in a disk subset of $\Sig'$. 
\end{definition} 

Suppose that $f:M \rightarrow J$ is a  sweep-out, and 
that $\Sig$ is a transversely oriented Heegaard surface of $M$. 
We say that $\Sig$ is a {\em spanning surface} for $f$ 
if there exist values $a,b \in \Int(J)$ such that 
$f^{-1}(a)$ is mostly above $\Sig$ and $f^{-1}(b)$ is mostly below $\Sig$. 
We say that $\Sig$ is a {\em splitting surface} for $f$ 
if it satisfies the following.  
First, there does not exist value $s \in \Int(J)$ such that 
$f^{-1}(s)$ is mostly above or mostly below $\Sig$. 
Second, $f|_\Sig$ is {\em almost Morse},  
that is, $f|_\Sig$ has only nondegenerate critical points and 
$f|_\Sig$ is Morse away from $-1$ and $1$, 
but there may be more than one minima and maxima at the levels  
$-1$ and $1$ respectively.  
We note that these definitions are coming from that in \cite[Definitions~11 and 12]{Jo10}. 

Proposition~27 in \cite{Jo10}, 
which will be used 
in the proof of Theorem~\ref{thm:homotopy_equivalence},  
can be stated in our term as follows:  

\begin{lemma}[Johnson {\cite[Proposition~27]{Jo10}}]\label{lem:splitting}
Let $f:M \rightarrow J$ be a sweep-out associated with a Heegaard splitting $(M,\Sig')$. 
If $f$ admits a splitting surface $\Sig$, 
then $d(\Sig') \le 2g(\Sig)$.  
\end{lemma}

Let $(M,\Sig)$ and $(M,\Sig')$ be Heegaard splittings of $M$. 
Assume that $f:M \rightarrow J$ is a sweep-out with $f^{-1}(0)=\Sig'$, and  
that $\{h_u:M \rightarrow J \mid u \in D^2\}$ is a family of sweep-outs 
associated with $(M,\Sig)$. 
Let $\Phi:M \times D^2 \rightarrow J^2 \times D^2$ be as in the previous subsection. 
Following \cite{Jo10}, 
let us consider the two subsets $\Rsc_a$ and $\Rsc_b$ of $J^2 \times D^2$ 
defined as follows: 
\[\Rsc_a:=\{(s,t,u) \in J^2 \times  D^2  \mid 
\mbox{$f^{-1}(s)$ is mostly above $h_u^{-1}(t)$ }\},\]
\[\Rsc_b:=\{(s,t,u) \in  J^2 \times D^2 \mid 
\mbox{$f^{-1}(s)$ is mostly below $h_u^{-1}(t)$}\}.\]
Here, for each $u \in D^2$ and $t \in J$, 
the transverse orientation of $h_u^{-1}(t)$ is 
determined by the sweep-out $h_u$. 
For example, if $t$ is sufficiently close to $-1$,  
then the point $(s,t,u)$ is in $\Rsc_a$ 
because $f^{-1}(s) \cap h_u^{-1}([-1,t])$ consists of finitely many  
properly embedded disks in the handlebody $h_u^{-1}([-1,t])$. 
Similarly, if $t$ is sufficiently close to $1$,  
then the point $(s,t,u)$ is in $\Rsc_b$. 
The regions $\Rsc_a$ and $\Rsc_b$ 
are nonempty open subset in $J^2 \times D^2$. 
The next proposition follows directly from the definition. 

\begin{proposition}\label{prop:regions}
The following hold: 
\begin{enumerate}
\item $\Rsc_a$ and $\Rsc_b$ are disjoint as long as $g(\Sig') \neq 0$. 
\item $\Rsc_a$ and $\Rsc_b$ are bounded by $\Gamma$. 
\item The regions $\Rsc_a$ and $\Rsc_b$ are convex in the $t$-direction,  
	that is, 
	if $(s,t,u)$ is in $\Rsc_a$ (resp. $\Rsc_b$),  
	then so is $(s,t',u)$ for any $t' \le t$ (resp. $t' \ge t$). 
\end{enumerate}
\end{proposition}

Set $J_u^2:=J^2 \times \{u\} \subset J^2 \times D^2$  for $u \in D^2$. 
Then,  for $u \in D^2$, 
the intersection $\Gamma \cap J_u^2 \subset J_u^2$ can be viewed  
as the (2D) graphic defined by sweep-outs $f$ and $h_u$. 
\begin{definition}
Let $f$ and $h_u$ as above. 
\begin{enumerate}
\item[(i)] We say that $h_u$ {\em spans} $f$ 
	if there exists $t \in J$ such that $h_u^{-1}(t)$ is a spanning surface for $f$. 
\item[(ii)] We say that $h_u$ {\em splits} $f$ 
	if there exists $t \in J$ such that $h_u^{-1}(t)$ is a splitting surface for $f$. 
\end{enumerate}
\end{definition}
\noindent
We also say that the graphic defined by $f$ and $h_u$ is {\em spanned} 
if $h_u$ spans $f$. 
Similarly, we say that the graphic defined by $f$ and $h_u$ is {\em split} 
if $h_u$ splits $f$. 

\begin{remark}\label{rem:distance}
By Lemma~\ref{lem:splitting}, 
the graphic defined by $f$ and $h_u$ cannot be split
if $d(\Sig') > 2g(\Sig)$. 
\end{remark}

Here are further remarks on the above definition.  
First, we remark that 
the condition (i) is equivalent to the following: 
there exists $t_0 \in J$ such that the horizontal segment $\{t=t_0\}$ in $J_u^2$ 
intersects both $\Rsc_a$ and $\Rsc_b$ 
(the left in the Figure~\ref{fig:spanned_split}). 
We also note that  
$J_u^2$ intersects $F_3$ transversely for $u \in D^2$, 
and hence $J_u^2 \cap F_3$ consists of finitely many open arcs. 
This is because $d(p_3 \circ \Phi)_w$ has the maximal rank for $w \in W$, 
where $p_3:J^2 \times D^2 \rightarrow D^2$ 
denotes the projection onto the third coordinate. 
Furthermore, after perturbing $\Phi$ if necessary, 
$J_u^2 \cap F_k$ consists of finitely many points for $0 \le k \le 2$ and $u \in D^2$. 
Under this assumption,  
the condition (ii) is equivalent to the following: 
there exists $t_0 \in J$ such that the horizontal segment $\{t=t_0\}$ in $J_u^2$ 
disjoint from both $\Cl(\Rsc_a)$ and $\Cl(\Rsc_b)$ 
(the right in the Figure~\ref{fig:spanned_split}). 

\begin{figure}
\includegraphics[width=12cm]{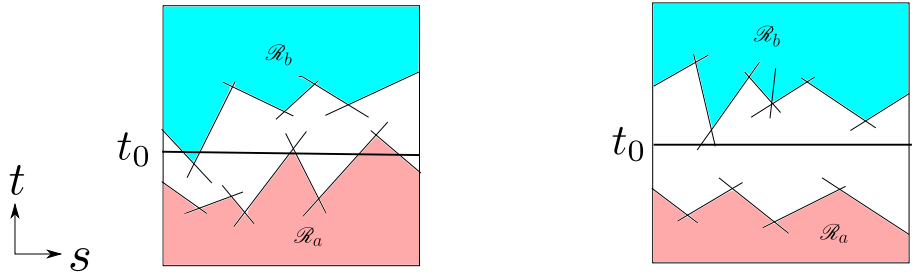}
\caption{The graphic defined by $f$ and $h_u$ is spanned 
if there exists a horizontal segment in $J_u^2$ 
intersecting both $\Rsc_a$ and $\Rsc_b$
(left). 
On the other hand,  
the graphic is split 
if there exists a horizontal segment 
disjoint from both $\Cl(\Rsc_a)$ and $\Cl(\Rsc_b)$ (right).}
\label{fig:spanned_split}
\end{figure}

\begin{proposition}\label{prop:frontiors} 
If $g(\Sig') \ge 2$, then 
$\Cl(\Rsc_a)$ and $\Cl(\Rsc_b)$ intersect only at points of $F_0$.  
\end{proposition}

\begin{proof}
We first note that by Proposition~\ref{prop:regions} (1) and (2), 
the intersection between $\Cl(\Rsc_a)$ and $\Cl(\Rsc_b)$ is contained in $\Gamma$. 
Suppose that 
$\Cl(\Rsc_a) \cap \Cl(\Rsc_b) \neq \emptyset$ 
and $y \in \Cl(\Rsc_a) \cap \Cl(\Rsc_b) \subset \Gamma$. 
Let $(s_0,t_0,u_0)$ be the coordinate of $y$. 
Let $l$ be the segment in $J^2 \times D^2$ defined by $l:=\{s=s_0\} \cap \{u=u_0\}$. 
Note that $y \in l$ by definition. 
Furthermore, it follows from Proposition~\ref{prop:regions} (3)  
that $l \subset \Cl(\Rsc_a) \cup \Cl(\Rsc_b)$. 
Consider a point $(s'_0,t_0,u'_0)$ 
obtained by perturbing the point $(s_0,t_0,u_0)$ in the $s$- and $u$-directions.  
We may assume that the segment $\wtil{l}:=\{s=s'_0\} \cap \{u=u'_0\}$ is transverse 
to each stratum of $\Gamma$. 
The preimage $\Phi^{-1}(\wtil{l})$ of $\wtil{l}$ 
is the genus $g(\Sig')$ Heegaard surface 
in $M \times \{u'_0\} \subset M \times D^2$, 
which can be naturally identified with $\Sig'$. 
Let $h: \Sig' \rightarrow J$ denote the function defined to be the restriction 
of $h_{u'_0}$ on $\Phi^{-1}(\wtil{l}) \cong \Sig'$. 
Then, $h$ is almost Morse. 

Now suppose, for the sake of contradiction, 
that $y$ is in $F_k$ with $k \ge 1$. 
Let $(s'_0,t_-,u'_0)$ denote the coordinate of the intersection point between $\wtil{l}$ 
and the boundary of $\Cl(\Rsc_a)$. 
Note that such a point is unique by Proposition~\ref{prop:regions} (3). 
Similarly, let $(s'_0,t_+,u'_0)$ denote the coordinate of the intersection point 
between $\wtil{l}$ and the boundary of $\Cl(\Rsc_b)$. 
Then, $h$ satisfies the following.  
\begin{itemize}
\item For any regular value $t \in J \setminus [t_-,t_+]$, 
	every loop of $h^{-1}(t)$ is trivial in $\Sig'$. 
\item The interval $[t_-,t_+]$ contains at most three critical values of $h$. 
\end{itemize}
\noindent 
It is easily seen that the Euler characteristic of such $\Sig'$ 
must be at least $-1$, 
but this is impossible because $g(\Sig') \ge 2$ by assumption. 
\end{proof}

\section{Proof of Theorem~\ref{thm:homotopy_equivalence}}\label{sec:poof}
In this section, we prove Theorem~\ref{thm:homotopy_equivalence}. 
Suppose that $M$ is a closed orientable $3$-manifold, and that
$(M,\Sig')$ is a genus $g(\Sig') \ge 2$ Heegaard splitting with $d(\Sig') > 2g(\Sig')+2$. 
Suppose that $(M,\Sig)$ is a once-stabilization of $(M,\Sig')$. 
Let $\He$ denote the path-connected component of $\He(M,\Sig)$ containing $\Sig$. 
Let $K^\pm$ and $\He^\pm \subset \He$ be as in Section~\ref{sec:introduction}. 
Set $\He^\cup:=\He^- \cup \He^+$ and 
$\He^\cap:=\He^- \cap \He^+$. 

By Theorem~\ref{thm:large_manifolds}, 
we have 
$\pi_k(\He^-)=\pi_k(\He^+)=\pi_k(\He^\cap)=0$ for $k \ge2$. 
By the Mayer-Vietoris exact sequence, 
it follows that $H_k(\He^{\cup};\mathbb{Z})=0$  for $k \ge 2$. 
Applying Hurewicz's theorem, 
we have $\pi_k(\He^\cup)=0$ for $k \ge 2$. 
On the other hand, by Theorems~\ref{thm:JM13} and \ref{thm:dist-hyp}, 
we have $\pi_k(\He)=0$ for $k \ge 2$. 
So, to prove Theorem~\ref{thm:homotopy_equivalence}, 
it is enough to show the following. 
 
\begin{lemma}\label{lem:pi1}
The inclusion $\He^{\cup} \rightarrow \He$ induces the isomorphism 
$\pi_1(\He^{\cup}) \rightarrow \pi_1(\He)$.
\end{lemma}
In \cite{Jo13}, Johnson proved that $\pi_1(\He)$ is generated by 
$\pi_1(\He^-)$ and $\pi_1(\He^+)$, 
and hence the induced map is a surjection. 
(In fact, 
using the notations in this paper, 
what he proved in \cite{Jo13} can be written as 
\[\pi_1(\He,\He^\cup)=1.\]
See Lemmas~2 and 3 in \cite{Jo13}. 
The following argument is motivated by this observation.) 
So in this paper, we focus on the proof of the injectivity of the induced map. 
In other words, we will show the following:  

\begin{lemma}\label{lem:2-homotopy}
The second homotopy group $\pi_{2}(\He,\He^\cup)$ of the pair $(\He,\He^\cup)$ 
vanishes. 
\end{lemma}

Let $e_0 \in \partial D^2$ be the basepoint. 
Let $\vph:(D^2,\partial D^2,e_0) \rightarrow (\He,\He^\cup,\Sig)$. 
We will show that $[\vph]=0 \in \pi_{2}(\He,\He^\cup)$. 
Let $f:M \rightarrow J$ be a sweep-out 
with $f^{-1}(0)= \Sig'$ and $f^{-1}(\pm1)=K^\pm$. 
By Lemma~\ref{lem:family_sweepout}, 
there exists a family   
$\{h_u:M \rightarrow J \mid u \in D^2\}$ of sweep-outs such that 
$h_u^{-1}(0)=\vph(u)$ for $u \in D^2$. 
The key of the proof is the following. 

\begin{lemma}\label{lem:spanning}
For any $u \in D^2$, the graphic defined by $f$ and $h_u$ is spanned. 
\end{lemma}

\begin{proof} 
Suppose, contrary to our claim, 
there exists $u_0 \in D^2$ such that 
the graphic defined by $f$ and $h_{u_0}$ is not spanned. 
Put $J_{u_0}^2:=\{(s,t,u) \in J^2 \times D^2 \mid u=u_0\}$. 
For brevity, we denote the restriction of $\Phi$ on $W=(M \times D^2) \setminus L$ by 
the same symbol $\Phi$ in the following. 
Set $\Gamma:=\Phi(S(\Phi))$.  
The intersection $\Gamma \cap J_{u_0}^2 \subset J_{u_0}^2$ is precisely the graphic 
defined by $f$ and $h_{u_0}$. 

As noted in Remark~\ref{rem:distance}, 
this graphic cannot be split. 
Then, there exists $t_0 \in J$ such that 
the horizontal segment $l:=\{t=t_0\} \subset J_{u_0}^2$ 
intersects both $\Cl(\Rsc_a)$ and $\Cl(\Rsc_b)$ at their boundaries  
(Figure~\ref{fig:intersection}). 
\begin{figure}
\includegraphics[width=12cm]{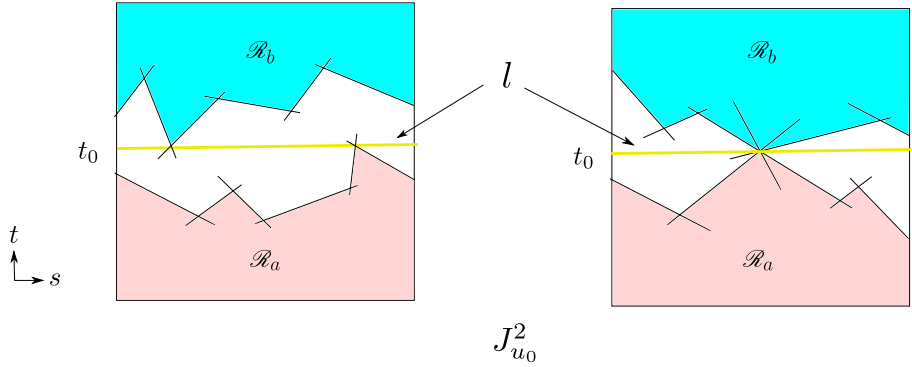}
\caption{If the graphic defined by $f$ and $h_{u_0}$ is not spanned, 
then there exists a horizontal segment $l \subset J_{u_0}^2$ 
intersecting $\Cl(\Rsc_a)$ and $\Cl(\Rsc_b)$ at their boundaries: 
the intersection is either separate vertices (left) or 
a single common vertex (right). 
In either case $l$ can be perturbed in $J^2 \times D^2$ so that $l$ does not meet 
$\Cl(\Rsc_a) \cup \Cl(\Rsc_b)$. }
\label{fig:intersection}
\end{figure}
By Proposition~\ref{prop:frontiors}, 
$\Cl(\Rsc_a)$ and $\Cl(\Rsc_b)$ 
intersect only at points of $F_0$. 
So,  pushing $l$ out of $J_{u_0}^2$ slightly, 
we get  an arc $\wtil{l} \subset J^2 \times D^2$  
such that 
\begin{itemize}
\item $\wtil{l}$ is disjoint from both 
	$\Rsc_a$ and $\Rsc_b$, and 
\item $\wtil{l}$ is transverse to each stratum of $\Gamma$. 
\end{itemize}
Furthermore, there is a family  
$\{l_t \mid t \in I\}$ of arcs with $l_0=l$ and $l_1=\wtil{l}$ such that 
for any $t \in (0,1]$,  
$l_t$ is transverse to each stratum of $\Gamma$. 
Note that $d(p_2 \circ \Phi)$ and 
$d(p_3 \circ \Phi)$ have the maximal ranks, 
where 
$p_2:J^2 \times D^2 \rightarrow J$ 
is the projection onto the second coordinate and 
$p_3:J^2 \times D^2 \rightarrow D^2$ is the projection onto the third coordinate. 
This means that $\Phi$ is transverse to $l=l_0$.  
In consequence, $\Phi$ is transverse to $l_t$ for all $t \in [0,1]$, 
and hence 
$\wtil{\Sig}:=\Phi^{-1}(\wtil{l})$ is a closed embedded surface in $M \times D^2$ 
isotopic to $\Sig_{t_0,u_0}:=\Phi^{-1}(l)$ ($=h^{-1}_{u_0}(t_0) \times \{u_0\}$). 
In particular, we have $g(\widetilde{\Sig})=g(\Sig')+1$. 

Let $q_1:M \times D^2 \rightarrow M$ denote the projection onto the first coordinate.  
Since the restriction $q_1|_{\Sig_{t_0,u_0}}:\Sig_{t_0,u_0} \rightarrow M$ 
is an embedding, 
so is $q_1|_{\wtil{\Sig}}:\wtil{\Sig} \rightarrow M$. 
We see that $q_1(\wtil{\Sig})$ is a splitting surface for $f$. 
Consider the restriction $f \circ q_1|_{\wtil{\Sig}}$ on $\wtil{\Sig}$. 
The arc $\wtil{l}$ intersects $\Gamma$ 
only at points in $F_3$, 
which correspond to fold points of $\Phi$. 
Thus, $f \circ q_1|_{\wtil{\Sig}}$ is almost Morse. 
Let $s \in J$ be any regular value of $f \circ q_1|_{\wtil{\Sig}}$. 
By definition, we can write 
\[(f \circ q_1|_{\wtil{\Sig}})^{-1}(s)
=(h^{-1}_u(t) \times \{u\}) \cap (f^{-1}(s) \times \{u\})
\subset M \times D^2\]  
for some $t \in J$ and $u \in D^2$. 
Since $\wtil{l}$ 
is disjoint from both $\Rsc_a$ and $\Rsc_b$, 
the preimage $(f \circ q_1|_{\wtil{\Sig}})^{-1}(s)$ contains at least one loop 
that is nontrivial in the surface $f^{-1}(s) \times \{u\}$. 
This implies that $q_1(\widetilde{\Sig})$ is a splitting surface for $f$. 
But it follows from Lemma~\ref{lem:splitting} that $d(\Sig') \le 2g(\wtil{\Sig})=2g(\Sig')+2$,   
and this contradicts the assumption. 
\end{proof}

We now return to the proof of Lemma~\ref{lem:2-homotopy}. 

\vspace{0.5em}
\noindent
\textit{Proof of Lemma~\ref{lem:2-homotopy}.} 
Let $p_2:J^2  \times D^2 \rightarrow J$ 
denote the projection onto the second coordinate. 
For $u \in D^2$, set $I_u:=p_2(\Cl(\Rsc_a)) \cap p_2(\Cl(\Rsc_b))$. 
Then, $t \in J$ is in $\Int(I_u)$ if and only if $h_u^{-1}(t)$ is a spanning surface for $f$. 
By Lemma~\ref{lem:spanning}, each $I_u$ is a nonempty subset of $J$. 
Furthermore, it follows from Proposition~\ref{prop:regions} (3) that 
each $I_u$ is a closed interval in $J$. 
So $\bigsqcup_{u \in D^2} I_u$ is an (trivial) $I$-bundle over $D^2$. 
Let $\sigma:D^2 \rightarrow \bigsqcup_{u \in D^2} I_u$ be a section of this $I$-bundle. 
Define $\wtil{\sigma}:D^2 \rightarrow \He$ by 
$\wtil{\sigma}(u):=h_u^{-1}(\sigma(u))$.
Recall that $\vph(u)=h_u^{-1}(0)$ for $u \in D^2$. 
The straight line homotopy connecting 
the $0$-section of $J^2 \times D^2 \rightarrow D^2$ to 
$\sigma$ induces the homotopy $\{\vph_r: D^2 \rightarrow \He \mid r \in [0,1]\}$  
with $\vph_0=\vph$ and $\vph_1=\wtil{\sigma}$. 
By Proposition~\ref{prop:regions} (3), 
we may choose $\sigma$ so that for $u \in \partial D^2$,  
$\vph(u)$ is isotopic to $\wtil{\sigma}(u)$ through surfaces 
disjoint from $K^-$ or $K^+$, 
depending on if $\vph(u) \in \He^+$ or $\vph(u) \in \He^-$ holds. 
This means that $\vph_r(u) \in \He^\cup$ for $u \in \partial D^2$ and $r \in [0,1]$. 
Clearly, $\{\vph_r\}$ can be chosen so that it preserves the basepoint. 
Thus, $\vph$ and $\wtil{\sigma}$ represent the same 
element of $\pi_2(\He,\He^\cup)$.  
Applying the homotopy described above, from now on,  
we may assume that the map $\vph$ satisfies the following: 
for any $u \in D^2$, $\Sig_u:=\vph(u)$ is a spanning surface for $f$. 

We think about fixed $u \in D^2$ for a moment. 
By assumption, there exist values $a, b \in J$ such that 
$\Sig^{\pp}:=f^{-1}(a)$ is mostly above $\Sig_u$, 
and $\Sig^{\pn}:=f^{-1}(b)$ is mostly below $\Sig_u$. 
By definition, every loop of $\Sig_u  \cap (\Sig^{\pp} \cup \Sig^{\pn})$ bounds a disk 
in $\Sig^{\pp} \cup \Sig^{\pn}$. 
The following observation is due to Johnson \cite{Jo13}.  

\begin{claim*}\label{clm:compression}
One of the two (possibly both) holds: 
\begin{enumerate}
\item Every loop in $\Sig_u  \cap \Sig^{\pp}$ bounds a disk in $\Sig_u$. 
\item Every loop in $\Sig_u  \cap \Sig^{\pn}$ bounds a disk in $\Sig_u$. 
\end{enumerate}
\end{claim*}

\begin{proof}
If we compress the surface $\Sig_u$ along innermost loops in 
$\Sig^{\pp} \cup \Sig^{\pn}$ repeatedly, 
we have a collection of surfaces disjoint from both $\Sig^{\pp}$ and $\Sig^{\pn}$. 
The point is that there is a surface $S$ in the collection 
that separates $\Sig^{\pp}$ from $\Sig^{\pn}$, 
and so $S$ is in the product region between $\Sig^{\pp}$ and $\Sig^{\pn}$.  
Note that such $S$ must have the genus at least $g(\Sig')$. 
This means that at most one of the two surfaces $\Sig^{\pp}$ and $\Sig^{\pn}$ 
contains an actual compression for $\Sig_u$ 
(i.e. a loop in $\Sig_u \cap \Sig^{\pp}$ or $\Sig_u \cap \Sig^{\pn}$ 
that is nontrivial in $\Sig_u$) 
because $g(\Sig_u)=g(\Sig')+1$. 
Therefore, either (1) or (2) holds.  
\end{proof}

Put $T'=\Sig^{\pp} \cup \Sig^{\pn}$. 
Take a loop $\ell$ in $\Sig_u \cap T'$ satisfying the following condition: 
\[\mbox{$\ell$ is trivial in $\Sig_u$ and $\ell$ is innermost in $T'$ 
among all the loops of $\Sig_u \cap T'$.} \eqno(*)\]
If we compress $\Sig_u$ along $\ell$ 
and  discard the sphere component,  
then the loop $\ell$ (and possibly some other loops in $\Sig_u  \cap T'$) is removed. 
Since $M$ is irreducible, 
this process can actually be achieved by an isotopy. 
Repeating this process as long as possible, 
all the loops in $\Sig_u \cap T'$ satisfying the condition ($*$)  
are finally removed. 
In particular, the resulting surface is disjoint from $\Sig^{\pp}$ or $\Sig^{\pn}$ 
depending on if (1) or (2) holds. 

We wish to do the above process simultaneously for $u \in D^2$. 
In fact, it is always possible using an argument in Hatcher \cite{Hat76}. 
The following is a sketch of the argument in \cite{Hat76}. 

We will construct 
a smooth family $\{\Theta_{u,r}:\Sig_u \rightarrow M \mid u \in D^2, r \in [0,1]\}$ 
of isotopies such that for any $u \in D^2$, 
$\Theta_{u,0}(\Sig_u)=\Sig_u$ and 
$\Theta_{u,1}(\Sig_u)$ is disjoint from either $K^-$ or $K^+$. 
By the above argument, 
we can see that there exist a finite cover $\{B_i\}$ of $D^2$ 
with $B_i \cong D^2$ and
a family $\{T'_i\}$ of (disconnected) surfaces with the following properties: 
\begin{itemize}
\item $T'_i$ is the union of the two level surfaces $\Sig^{\pp}_i$ and $\Sig^{\pn}_i$ 
	of $f$. 
\item If $u \in B_i$, then 
	$\Sig^{\pp}_i$ is mostly above $\Sig_u$, and 
	$\Sig^{\pn}_i$ is mostly below $\Sig_u$.  
\end{itemize} 

For $u \in D^2$,  
let $\wtil{\mathC}_u$ be the set of intersection loops between  
$\Sig_u$ and $\bigcup_i T'_i$, 
where the union is taken over all $i$'s such that 
$u \in B_i$.  
We denote by $D'(\ell)$ the disk in $T'_i$ bounded by $\ell$ for $\ell \in \wtil{\mathC}_u$. 
(Note that such a disk is unique 
because each component of $T'_i$ is not homeomorphic to $S^2$.) 
For $u \in D^2$, 
let $\mathC_u$ be the subset of $\wtil{\mathC}_u$ consisting of those loop $\ell$ 
such that $\ell$ is trivial in $\Sig_u$, and that 
$D'(\ell)$ contains no other intersection loop that is nontrivial in $\Sig_u$. 
Furthermore, for $\ell \in \mathC_u$, 
we denote by $D(\ell)$ the disk in $\Sig_u$ bounded by $\ell$. 
(Again, note that such a disk is unique 
because $\Sig_u \neq S^2$.) 
For $u \in D^2$, we define the partial order $<'$ on $\mathC_u$ by 
\[\ell <' m \Leftrightarrow D'(\ell) \subset D'(m).\] 

Let $\{B'_i\}$ be a finite cover of $D^2$ obtained by shrinking each $B_i$ slightly 
so that $B'_i \subset \Int(B_i)$ for $i$. 
Take a family $\{\alpha_u:\mathC_u \rightarrow (0,2) \mid u \in D^2\}$ of functions  
with the following properties: 
\begin{itemize}
\item If $\ell,m \in \mathC_u$ and $\ell <' m$, then $\alpha_u(\ell) <\alpha_u(m)$.  
\item If $u \in B'_i$ and $\ell \subset \Sig_u \cap T'_i$, 
	then $\alpha_u(\ell) <1$. 
\item If $u \in \partial B_i$ and $\ell \subset \Sig_u \cap T'_i$, 
	then $\alpha_u(\ell) > 1$. 
\end{itemize}
The function $\alpha_u$ shows the times 
when intersection loops that belong to $\mathC_u$ are eliminated by compressing. 
Let $G$ denote the union of the images 
of $\alpha_u$'s in $D^2 \times [0,2]$. 
Note that for each intersection loop $\ell$, 
the images of the loops corresponding to $\ell$ 
form a 2D sheet over some $B_i$, 
and so $G$ can be written as the union of these sheets. 
We can view $G$ as a ``chart" to compress the surface $\Sig_u$:  
if we compress $\Sig_u$ following this chart upward from $r=0$ to $r=2$, 
then we get the sequence of surfaces. 
Note that the following subtle case may occur:   
if $\ell$ and $m$ are loops of $\Sig_u \cap T'_i$ 
with $D(\ell) \subset D(m)$ and $\alpha_u(m)< \alpha_u(\ell)$, 
then the loop $\ell$ is eliminated automatically before the time $\alpha_u(\ell)$. 
This example shows that we should use the ``reduced" chart $\widehat{G}$ 
rather than $G$, which is obtained from $G$ 
by removing the parts of the sheets corresponding to such $\ell$'s. 

For every $u$, we will define the isotopy $\Theta_{u,r}$ as follows.  
Let $N(\widehat{G})$ denote a small fibered neighborhood of $\widehat{G}$. 
The interval $\{u\} \times [0,2]$ intersects $N(\widehat{G})$ at its subintervals 
$J^{(k)}_u$, where $1 \le k \le n=n(u)$. 
Define $\wtil{\Theta}_{u,r}$ to be the isotopy obtained 
by piecing together the isotopies $\theta^{(1)}_{u,r},\ldots,\theta^{(n)}_{u,r}$ 
in the way suggested by $\widehat{G}$. 
Here each $\theta^{(k)}_{u,r}$ is an isotopy with its $r$-support in $J^{(k)}_u$, 
and corresponds to the compression along a loop in $\mathC_u$. 
See Figure~\ref{fig:chart}. 
Now we can define $\Theta_{u,r}$ as the restriction of $\wtil{\Theta}_{u,r}$ on $[0,1]$.  
\begin{figure}
\includegraphics[width=7cm]{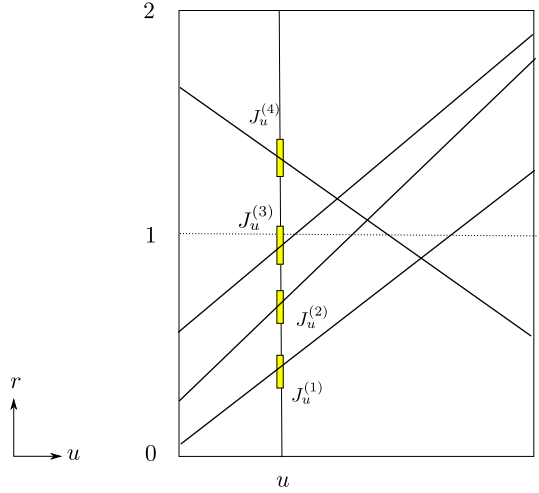}
\caption{The isotopy $\wtil{\Theta}_{u,r}$ is obtained 
by piecing together the isotopies $\theta^{(1)}_{u,r}$, $\theta^{(2)}_{u,r}$ 
$\theta^{(3)}_{u,r}$ and $\theta^{(4)}_{u,r}$. 
(In this example, 
$\wtil{\Theta}_{u,r}$ can be written as the concatenation 
$\theta^{(1)}_{u,r}*\theta^{(2)}_{u,r}*\theta^{(3)}_{u,r}*\theta^{(4)}_{u,r}$ 
of the small isotopies.) 
The $r$-support of $\theta^{(k)}_{u,r}$ is contained in $J^{(k)}_u$.}
\label{fig:chart}
\vspace{-1em}
\end{figure}

It remains to see that we can modify the above construction 
to get the isotopy $\{\Theta_{u,r}\}$ to be smooth for $u \in D^2$. 
It is enough to show that each factor $\theta^{(k)}_{u,r}$ of $\Theta_{u,r}$ 
can be chosen so that it varies smoothly for $u$. 
For simplicity, we will think about the isotopy $\theta^{(1)}_{u,r}$ in the following 
although the same argument applies to any $\theta^{(k)}_{u,r}$. 
The isotopy $\theta^{(1)}_{u,r}$ corresponds 
to the compression along a loop $\ell_u \in \mathC_u$ for each $u$. 
Assume that $\ell_u \subset T'_i$ for any $u$. 
Denote by $D^3(\ell_u)$ the $3$-ball in $M$ bounded by the $2$-sphere 
$D(\ell_u) \cup D'(\ell_u)$. 
(Note that such a $3$-ball is unique because $M \neq S^3$.) 
Let $(D^3,D,D')$ be the standard triple of disks, that is, 
$D$ and $D'$ are the upper and lower hemispheres  
in the boundary $\partial D^3$ of the standard $3$-ball $D^3$, respectively. 
There is an identification $\phi_u:(D^3(\ell_u),D(\ell_u),D'(\ell_u)) \rightarrow (D^3,D,D')$ 
for every $u$. 
Then the arguments in \cite{Hat76} together with the Smale Conjecture 
(:the space $\Diff(D^3 \,\mathrm{rel}.\, \partial D^3)$ is contractible), 
which is proved in \cite{Hat83}, 
show that $\phi_u$ can be chosen 
so that it varies smoothly for $u$. 
\footnote[2]{More specifically, 
we need the arguments at the end of \S1 in \cite{Hat76}, 
where the sought isotopy, denoted by $h_{tu}$ in that paper, 
is constructed.  
It starts by taking a suitable triangulation of $D^n$ 
and then proceeds by extending the isotopy over the $k$-skeleton inductively. 
The homotopy group $\pi_k(\Diff(D^3 \,\mathrm{rel}.\, D))$ appears  
as an obstruction to extending a map. 
(As we work in the smooth category, we use the Smale Conjecture 
instead of the Alexander trick.)}
Now we can define $\theta^{(1)}_{u,r} \equiv \phi_u^{-1} \circ F_r \circ \phi_u$ 
on $D(\ell_u) \subset \Sig_u$ and $\theta^{(1)}_{u,r} \equiv \Theta_{u,0}$ 
on the complement of a small neighborhood of $D(\ell_u)$ in $\Sig_u$, 
where $\{F_r:D^3 \rightarrow D^3 \mid r \in [0,1]\}$ 
is an isotopy that carries $D$ to $D'$ across $D^3$. 
Therefore, it follows that $\{\Theta_{u,r}\}$ is smooth for $u \in D^2$. 

Finally, we see that $\Theta_{u,1}(\Sig_u) \in \He^\cup$ for  $u \in D^2$. 
Let $u$ be any point in $D^2$. 
Take a path $\rho:[0,1] \rightarrow D^2$ 
with $\rho(0)=e_0$ and $\rho(1)=u$. 
It suffices to show that 
the path $\wtil{\rho}:[0,1] \rightarrow \He$ defined by 
$\wtil{\rho}(t):=\Theta_{\rho(t),1}(\Sig_{\rho(t)})$ 
is wholly contained in $\He^\cup$. 

For brevity, we denote by $\Sig_t$ the surface $\Theta_{\rho(t),1}(\Sig_\rho(t))$ 
for $t \in [0,1]$ in the following. 
The cover $\{B_i\}$ of $D^2$ induces the cover 
$\{I_k \mid 0 \le k \le n\}$ of $[0,1]$ by finitely many closed intervals. 
By passing to a subcover if necessary, 
we may assume that $I_k \cap I_j=\emptyset$ if $|k-j| > 1$. 
As we have seen above, 
there exists a family $\{\Sig^{\pp}_k \cup \Sig^{\pn}_k\}$ ($=\{T'_k\}$)  
of level surfaces of $f$ and the following hold: 
\begin{itemize}
\item $\Sig^{\pp}_k$ is mostly above $\Sig_t$ if $t \in I_k$. 
	Similarly, $\Sig^{\pn}_k$ is mostly below $\Sig_t$ if $t \in I_k$.
\item For each $k$, one of the two surfaces $\Sig^{\pp}_k$ and $\Sig^{\pn}_k$ 
	is disjoint from $\Sig_t$ if $t \in I_k$. 
\end{itemize}
\noindent 
As is naturally expected, 
the following holds:  

\begin{claim*}
Suppose that $t \in I_k$. 
If $\Sig_t \cap \Sig^{\pp}_k=\emptyset$ and $\Sig_t \cap \Sig^{\pn}_k \neq \emptyset$, 
then $\Sig_t \in \He^-$. 
Similarly, 
if $\Sig_t \cap \Sig^{\pn}_k=\emptyset$ and $\Sig_t \cap \Sig^{\pp}_k \neq \emptyset$, 
then $\Sig_t \in \He^+$. 
\end{claim*}

\begin{proof}
The proof is by induction on $k$. 
The following argument is based on the idea in \cite{Jo13}. 
By definition, $\Sig_t \in \He^\cap$ for $t \in I_0$. 
Thus our claim holds on $I_0$. 
So, in what follows, 
we assume that $k>0$ 
and that our claim holds on any interval $I_j$ with $0 \le j<k$. 

Let $t \in I_k$. 
Without loss of generality, 
we may assume that 
$\Sig_t \cap \Sig^{\pp}_k=\emptyset$ and $\Sig_t \cap \Sig^{\pn}_k \neq \emptyset$. 
Fix $t_0 \in I_k \cap I_{k-1}$. 
Note that 
$\Sig_{t_0}$ and $\Sig_t$ are isotopic through surfaces disjoint from $K^+$. 
Thus it is enough to show that $\Sig_{t_0} \in \He^-$. 
There are three cases to consider. 

\vspace{0.5em}
\noindent 
\textbf{Case~1:} 
$\Sig_{t_0} \cap \Sig^{\pp}_{k-1}=\emptyset$ 
and $\Sig_{t_0} \cap \Sig^{\pn}_{k-1} \neq \emptyset$ hold. 

By the assumption of induction,  
this implies that $\Sig_{t_0} \in \He^-$ and our claim holds in this case. 

\vspace{0.5em}
\noindent 
\textbf{Case~2:} 
$\Sig_{t_0} \cap \Sig^{\pp}_{k-1} \neq \emptyset$ 
and $\Sig_{t_0} \cap \Sig^{\pn}_{k-1}=\emptyset$ hold. 

We will see 
that $\Sig_{t_0} \in \He^\cap$. 
This is same as saying 
that $\Sig_{t_0}$ is a Heegaard surface 
of $M \setminus \Int(N(K^+ \cup K^-)) \cong \Sig' \times J$,  
where $N(K^+ \cup K^-)$ is a sufficiently small neighborhood of $K^+ \cup K^-$.  
Since $\Sig_{t_0} \cap \Sig^{\pn}_{k-1}=\Sig_{t_0} \cap \Sig^{\pp}_k=\emptyset$, 
$\Sig_{t_0}$ separates $K^+$ from $K^-$. 
First, we see that 
$\Sig_{t_0}$ is bicompressible in $M \setminus (K^+ \cup K^-)$. 
By assumption, 
there exists a loop $\ell \subset \Sig_{t_0} \cap \Sig^{\pp}_{k-1}$ 
bounding a disk $D^- \subset \Sig^{\pp}_{k-1}$ 
such that $\ell$ is nontrivial in $\Sig_{t_0}$. 
Similarly, there exists a loop $m \subset \Sig_{t_0} \cap \Sig^{\pn}_{k}$ 
bounding a disk $D^+ \subset \Sig^{\pn}_{k}$ 
such that $m$ is nontrivial in $\Sig_{t_0}$. 
Since $D^-$ and $D^+$ are in the opposite side of $\Sig_{t_0}$ to each other, 
$\Sig_{t_0}$ is bicompressible in $M \setminus (K^+ \cup K^-)$.

It is known that any genus $g(\Sig')+1$ bicompressible surface 
in $\Sig' \times J$ separating $\Sig' \times \{1\}$ from $\Sig' \times \{-1\}$ 
must be reducible (cf. \cite{Jo13}). 
This means that there exists a $2$-sphere 
$P \subset M \setminus (K^+ \cup K^-)$ 
intersecting $\Sig_{t_0}$ at a single nontrivial loop in $\Sig_{t_0}$. 
Since $M$ is irreducible, 
$P$ cuts $(M,\Sig_{t_0})$ into 
the two Heegaard splittings: one is a genus $g(\Sig')$ Heegaard splitting of $M$ 
and the other is a genus $1$ Heegaard splitting of $S^3$. 
If we denote by $S$ the genus $g(\Sig')$ surface obtained by
cutting $\Sig_{t_0}$ along $P$, 
then $S$ still separates $K^+$ from $K^-$. 
Thus, $S$ is isotopic to $\Sig'$ in the complement of $K^+ \cup K^-$. 
This shows that $\Sig_{t_0}$ is a genus $g(\Sig')+1$ Heegaard surface 
in $M \setminus \Int(N(K^+ \cup K^-))$. 
Therefore, we conclude that $\Sig_{t_0} \in \He^\cap$ in this case. 

\vspace{0.5em}
\noindent 
\textbf{Case~3:} 
$\Sig_{t_0} \cap \Sig^{\pp}_{k-1}=\Sig_{t_0} \cap \Sig^{\pn}_{k-1}=\emptyset$ hold. 

Let $j$ denote the minimal integer with the following property: 
for any $j<j' \le k$ and $t \in I_{j'}$, 
$\Sig_t \cap \Sig^{\pp}_{j'}=\Sig_t \cap \Sig^{\pn}_{j'}=\emptyset$ hold. 
If $j=0$, then $\Sig_{t_0}$ is isotopic to $\Sig_0$ 
through surfaces disjoint from $K^+ \cup K^-$. 
This shows that $\Sig_{t_0} \in \He^\cap$. 
So we may assume that $j>0$ in the following. 
Let $t_1 \in I_j \cap I_{j+1}$. 

First, we assume that $\Sig_{t_1} \cap \Sig^{\pp}_j=\emptyset$ 
and that $\Sig_{t_1} \cap \Sig^{\pn}_j \neq \emptyset$. 
By the assumption of induction, 
it follows that $\Sig_{t_1} \in \He^-$. 
Since $\Sig_{t_1}$ and $\Sig_{t_0}$ are 
isotopic in $M \setminus (K^+ \cup K^-)$, 
we have $\Sig_{t_0} \in \He^-$ in this case. 

Next, we assume that $\Sig_{t_1} \cap \Sig^{\pp}_j \neq \emptyset$ 
and that $\Sig_{t_1} \cap \Sig^{\pn}_j=\emptyset$. 
Then, there exists a compression disk $D^- \subset \Sig^{\pp}_j$ for $\Sig_{t_1}$. 
Since $\Sig_{t_1}$ and $\Sig_{t_0}$ are 
isotopic in $M \setminus (K^+ \cup K^-)$, 
$\Sig_{t_0}$ has a compression disk disjoint from $K^-$ as well. 
On the other hand, as we have seen above, 
$\Sig^{\pn}_k$ contains a compression disk $D^+$ for $\Sig_{t_0}$ 
lying in the opposite side of $\Sig_{t_0}$ to $D^-$. 
Thus, $\Sig_{t_0}$ is bicompressible in $M \setminus (K^+ \cup K^-)$. 
Now applying the same argument as in Case~2, 
we have $\Sig_{t_0} \in \He^\cap$ and this completes the proof. 
\end{proof}

The above claim implies that 
the image of $\wtil{\rho}:[0,1] \rightarrow \He$ 
is contained in $\He^\cup$. 
In particular, we have $\Theta_{u,1}(\Sig_{u}) \in \He^\cup$. 
Therefore, we conclude that $[\vph]=0 \in \pi_2(\He,\He^\cup)$ 
and this finishes the proof of Lemma~\ref{lem:2-homotopy}. 
\qed

\section
{The isotopy subgroup of a Heegaard splitting of a handlebody}
\label{sec:conclusion}

\subsection{Proof of Theorem~\ref{thm:handlebody}}
\label{sub:handlebody}

We now give a proof of Theorem~\ref{thm:handlebody}.  
Let $V$ be a genus $g(V) \ge 2$ handlebody, and 
let $(V,\Sig)$ be a genus $g(V)+1$ Heegaard splitting of $V$. 
Fix a complete system $E_1,\ldots,E_{g(V)}$ of meridian disks for $V$.   
Consider a  properly embedded, boundary parallel arc $I$ in $V$ 
that is disjoint from $\bigcup^{g(V)}_{i=1} E_i$. 
The surface $\Sig$ can be viewed as the boundary of 
a small neighborhood $N(\partial V \cup I)$ of $\partial V \cup I$.  
In the same spirit in \cite{JM13},
we define the space $\Unk(V,I)$ of unknotting arcs to be $\Diff(V)/\Diff(V,I)$. 
Then, the following holds: 
\begin{theorem}[Scharlemann {\cite[Theorem~5.1]{Sc13}}]\label{thm:unknot}
The group $\Isot(V,\Sig)$ is isomorphic to $\pi_1(\Unk(V,I))$. 
\end{theorem}
\noindent
Thus, it suffices to show that  $\pi_1(\Unk(V,I))$ is finitely presented. 
 
Fix a parallelism disk $E$ for $I$ disjoint from  
$\bigcup_{i=1}^{g(V)} E_i$. 
Furthermore, fix a spine $K$ of $V$ such that 
$K \cap E=\emptyset$ and 
$K$ intersects each $E_i$ at a single point. 
We now consider the two subspaces of $\Unk(V,I)$ defined as follows: 
\[U_1:=\left\{I' \in \Unk(V,I) \mid 
\mbox{$I'$ admits a parallelism disk $E'$ with $E' \cap K=\emptyset$}\right\},\]
\[U_2:=\left\{I' \in \Unk(V,I) \mid I' \cap \bigcup_{i=1}^{g(V)} E_i=\emptyset\right\}.\]

\noindent Note that $U_1$, $U_2$ and $U_1 \cap U_2$ are all connected. 

The group $\pi_1(U_1)$ is 
identical to the group $\Ff_{E}$ in \cite{Sc13}, 
which is called the {\em freewheeling} subgroup in that paper.  
This group is an extension of $\pi_1(\partial V)$ by $\mathbb{Z}$, 
and generated by $\lambda_i$, $\mu_i$ ($1 \le i \le g(V)$) and $\rho$ 
shown in Figure~\ref{fig:wheeling}. 
For each $i$, 
$\lambda_i$ is represented by an isotopy of parallelism disk $E$ 
along a longitudinal loop that intersects $\partial E_i$ at a single point. 
Similarly, $\mu_i$ is represented by an isotopy of the parallelism disk $E$ 
along a meridional loop corresponding to $\partial E_i$. 
The set $\{\lambda_i,\mu_i \mid 1 \le i \le g(V)\}$ 
corresponds to a generating set of $\pi_1(\partial V)$. 
$\rho$ is defined to be the half rotation of the parallelism disk $E$. 
Let $P$ denote the planar surface obtained by cutting $\partial V$ along 
simple closed curves $\partial E_1,\ldots,\partial E_{g(V)}$.  
Then, the group $\pi_1(U_2)$ is isomorphic to 
the $2$-braid group $B_2(P)$ of $P$. 
Following \cite{Sc13}, we define 
the {\em anchored} subgroup $\Af_{E_1,\ldots,E_{g(V)}}$ 
of $\pi_1(U_2)$ as follows. 
This is generated by $2g(V)$ elements $\alpha_i$, $\alpha'_i$ ($1 \le i \le g(V)$)
shown in Figure~\ref{fig:anchor}.  
Here each of $\alpha_i$ and $\alpha'_i$ is represented by an isotopy of $I$ 
that moves the one endpoint $p_1$ of $I$ along a meridional loop and 
fixes the other endpoint $p_0$. 
Note that we can write $\alpha'_i=\lambda_i^{-1}\alpha_i\lambda_i$ 
as elements of $\pi_1(\Unk(V,I))$. 
The group $\pi_1(U_2)$ is generated by $\Af_{E_1,\ldots,E_{g(V)}}$ and $\rho$. 
\begin{figure}
\includegraphics[width=9cm]{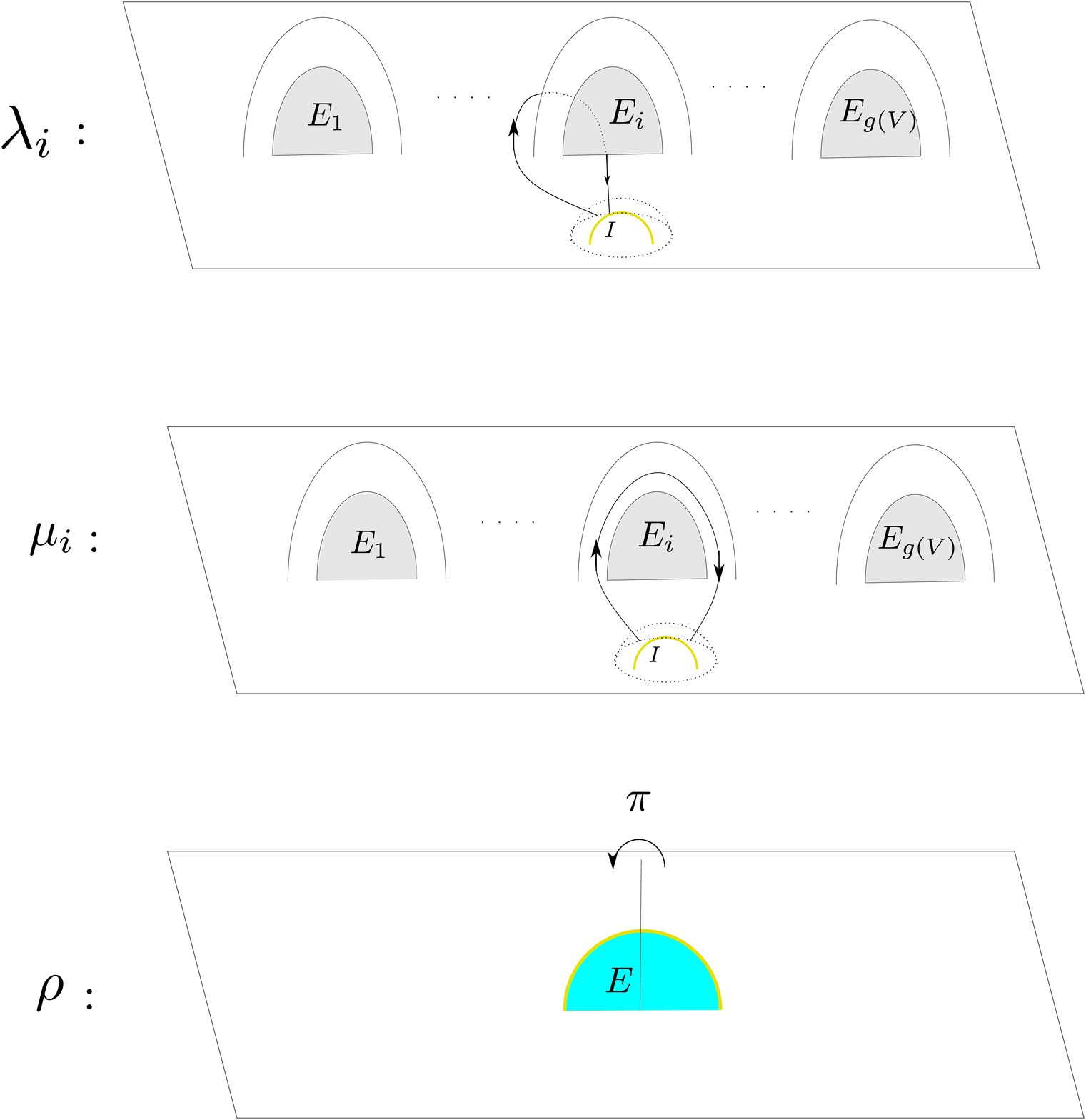}
\caption{The group $\pi_1(U_1)$ is generated by $2g(V)+1$ elements.}
\label{fig:wheeling}
\end{figure}

\begin{figure}
\includegraphics[width=16cm]{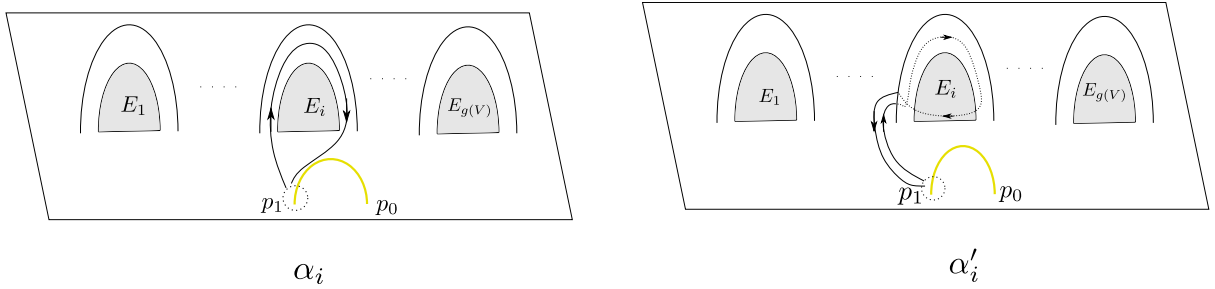}
\caption{The group $\Af_{E_1,\ldots,E_{g(V)}}  \subset \pi_1(U_2)$ 
is generated by $2g(V)$ elements.}
\label{fig:anchor}
\end{figure}

The groups $\pi_1(U_1)$, $\pi_1(U_2)$ and $\pi_1(U_1 \cap U_2)$ are all finitely presented. 
By van Kampen's theorem, the proof is finished 
if the following is shown: 

\begin{lemma}\label{lem:factorization}
The inclusion $U_1 \cup U_2 \rightarrow \Unk(V,I)$ is a homotopy equivalence. 
\end{lemma}

\noindent 
In fact, by the same argument as in Section~\ref{sec:poof}, 
it is easily seen that $\pi_k(U_1 \cup U_2)= \pi_k(\Unk(V,I))=0$ for $k \ge 2$. 
(And of course, this fact is unnecessary for our present purpose.) 
So we will see that 
the natural map $\pi_1(U_1 \cup U_2) \rightarrow \pi_1(\Unk(V,I))$ 
is an isomorphism. 

\begin{proof}
For brevity, set $U:=\Unk(V,I)$. 
In \cite{Sc13}, it was shown that $\pi_1(U)$ is generated by the two subgroups 
$\pi_1(U_1)$ ($=\Ff_E$) and $\Af_{E_1,\ldots,E_{g(V)}}$ ($\subset \pi_1(U_2)$). 
It follows from this fact 
that the map $\pi_1(U_1 \cup U_2) \rightarrow \pi_1(U)$ is a surjection. 
We will see that the map $\pi_1(U_1 \cup U_2) \rightarrow \pi_1(U)$ is an injection. 
In other words, we will see that $\pi_2(U,U_1 \cup U_2)=0$. 
Let $\vph:(D^2,\partial D^2) \rightarrow (U,U_1 \cup U_2)$. 
Put $I_u:=\vph(u)$ for $u \in D^2$. 
In the same spirit of the proof of Lemma~\ref{lem:family_sweepout}, 
we can show the following. 

\begin{claim}\label{clm:family_parallelism}
There exists a (smooth) family of disks $\{E_u \mid u \in D^2\}$ in $V$ 
such that $E_u$ is a parallelism disk for $I_u$. 
\end{claim} 

\begin{proof}
By \cite{Sc13}, the map $\Diff(V) \rightarrow \Diff(V)/\Diff(V,\Sig)$ 
is homotopy equivalent to 
$\Diff(V) \rightarrow \Diff(V)/\Diff(V,I)$. 
The former is a fibration \cite{JM13}, and so is the latter. 
Thus, the map  $\vph:D^2 \rightarrow U$ lifts to a map 
$\wtil{\vph}:D^2 \rightarrow \Diff(V)$.  
Now define $E_u:=\wtil{\vph}(u)(E)$. 
\end{proof}

Since $\vph(\partial D^2) \subset U_1 \cup U_2$, 
the isotopy $\{I_u \mid u \in \partial D^2\}$ represents an element of 
$\pi_1(U_1 \cup U_2)$.  
So we can write this isotopy as a product 
$\omega_1 \om_2 \cdots \om_n$ of the $\om_k$'s, 
where each $\om_k$ is either 
$\lambda_i$, $\mu_i$, $\rho$, $\alpha_i$,$\alpha'_i$ or their inverses. 
Corresponding to this factorization, 
there is a division of $\partial D^2$ into the intervals 
$J_1:=[u_0,u_1],\ldots,J_n:=[u_{n-1},u_n]$ with $u_0=u_n$. 

\begin{claim}\label{clm:intersection_pattern}
After a deformation of $\{E_u \mid u \in D^2\}$ near $\partial D^2$, 
the following hold for any $u \in \partial D^2$:  
\begin{enumerate}
\item[(i)] $E_u$ intersects $\bigcup_{i=1}^{g(V)} E_i$ at finitely many arcs, 
	and $E_u$ intersects $K$ at finitely many points. 
\item[(ii)] Each arc of $E_u \cap \bigcup_{i=1}^{g(V)} E_i$ is parallel to $I_u$ 
	in $E_u$. 
\item[(iii)] If $a$ and $a'$ are arcs of  
	$E_u \cap \bigcup_{i=1}^{g(V)} E_i$, 
	then $a$ and $a'$ are nested in the following sense: 
	if $\Delta$ and $\Delta'$ are bigons in $E_u$ cut by $a$ and $a'$ respectively, 
	then either $\Delta \subset \Delta'$ or $\Delta' \subset \Delta$ holds.  
\end{enumerate}
\end{claim}
\noindent See Figure~\ref{fig:parallelism}. 
\begin{figure}
\includegraphics[width=6cm]{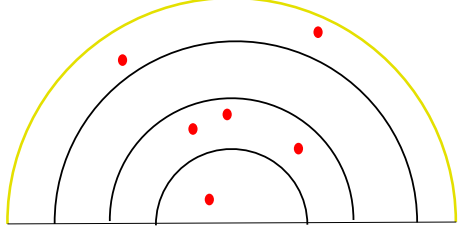}
\caption{For $u \in \partial D^2$, 
the parallelism disk $E_u$ intersects $\bigcup_{i=1}^{g(V)} E_i$ 
only at arcs parallel to $I_u$  
and intersects $K$ at finitely many points.}
\label{fig:parallelism}
\end{figure}

\begin{proof}

The key is the following simple observation.  
For each interval $J_k$, there are the three possibilities: 

\begin{itemize}
\item $\om_k=\lam_i^\ep$ for some $1 \le i \le g(V)$ and $\ep=\pm1$. 
	Then, during the move $\om_k$, some intersection arcs of  
	$E_u \cap \bigcup_{i=1}^{g(V)} E_i$ are introduced or removed 
	(possibly both may occur).
	All such arcs are parallel to $I_u$ in $E_u$.  
	The intersection pattern of $E_u \cap K$ is not changed by $\om_k$. 
	See the above in Figure~\ref{fig:intersection_arcs}. 
\item $\om_k=\alpha_i^\ep$ 
	or $\om_k=\alpha_i^{\prime\ep}$ for some $1 \le i \le g(V)$ and $\ep=\pm1$. 
	Then, during the move $\om_k$, a single intersection point of $E_u \cap K$ 
	is introduced or removed. 
	The intersection pattern of $E_u \cap \bigcup_{i=1}^{g(V)} E_i$ 
	is not changed by $\om_k$. 
	See the below in Figure~\ref{fig:intersection_arcs}. 
\item $\om_k=\mu_i^\ep$ or $\om_k=\rho^\ep$ for some $1 \le i \le g(V)$ 
	and $\ep=\pm1$. 
	Then, during the move $\om_k$, the intersection pattern of 
	$E_u \cap  (\bigcup_{i=1}^{g(V)} E_i \cup K)$ does not change. 
\end{itemize}
Recall that $E_{u_0}=E_{u_n}=E$ by definition. 
In particular, $E_{u_0} \cap (\bigcup_{i=1}^{g(V)} E_i \cup K)=\emptyset$. 
By the above observation,
it follows that the conditions (i), (ii) and (iii) are satisfied 
on the interval $J_1$. 
By an inductive argument, 
we can see that these three conditions are satisfied 
on any interval $J_k$ as well.  
\end{proof}
\begin{figure}
\includegraphics[width=14cm]{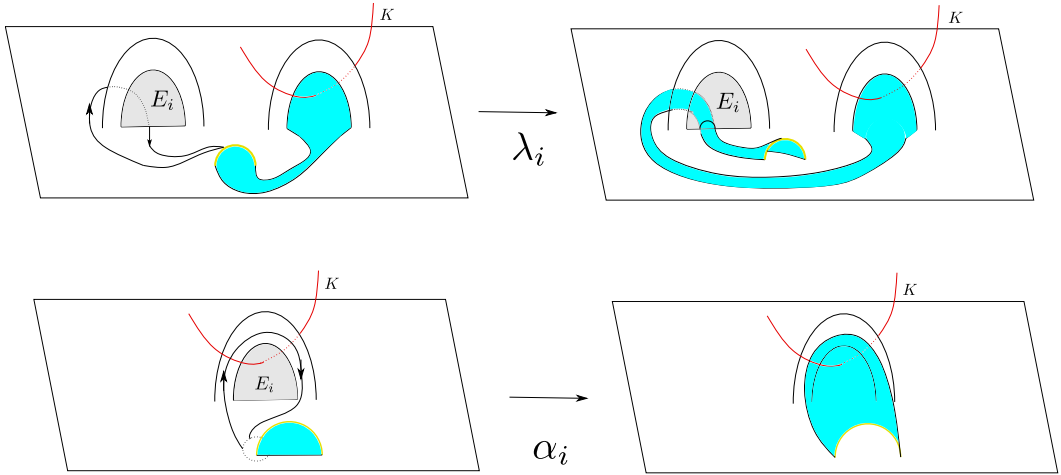}
\caption
{The move $\lambda_i$ introduces or removes 
arcs parallel to $I_u$ (above), and 
the move $\alpha_i$ introduces or removes 
a single point (below).} 
\label{fig:intersection_arcs}
\end{figure}

Put 
$B:=\{re^{\sqrt{-1}\theta} \in \mathbb{C} \mid 0 \le r \le 1,\, 0 \le \theta \le \pi\}$. 
There is a smooth family $\{f_u:E_u \rightarrow B \mid u \in D^2\}$ 
of diffeomorphisms between $E_u$ and $B$. 
(More rigorously, this is a consequence of the fact that 
the space $\Diff(D^2 \,\mathrm{rel}.\,\partial D^2)$ is contractible \cite{Sm59}.)
Furthermore, by Claim~\ref{clm:intersection_pattern}, 
we may choose $\{f_u\}$ so that it satisfies the following:  
for any $u \in \partial D^2$, each arc of $E_u \cap \bigcup_{i=1}^{g(V)} E_i$ is mapped 
to an arc written as 
$\{re^{\sqrt{-1} \theta} \in B \mid r=r_0, 0 \le \theta \le \pi\}$ 
for some $0<r_0\le 1$. 
For $t \in [0,1)$, 
define $\sigma_t:B \rightarrow B$ by 
$\sigma_t(re^{\sqrt{-1}\theta}):=(1-t)re^{\sqrt{-1}\theta}$. 
Set $\Theta_{u,t}:=f_u^{-1} \circ \sigma_t \circ f_u$ for $u \in D^2$ and $t \in [0,1)$. 
Then, the isotopy $\Theta_{u,t}$ shrinks $I_u$ along $E_u$
into a small neighborhood of a point in $E_u \cap \partial V$ as $t \rightarrow 1$. 
If $t$ is sufficiently close to $1$, then 
$\Theta_{u,t}(I_u) \in U_1$. 
Furthermore, 
by definition, for $u \in \partial D^2$ and $t \in [0,1)$, 
$\Theta_{u,t}(I_u)$ is disjoint from either $K$ or $\bigcup_{i=1}^{g(V)} E_i$. 
Let $u \in \partial D^2$ and $t \in [0,1)$. 
If $\Theta_{u,t}(I_u) \cap K=\emptyset$, then $\Theta_{u,t}(I_u) \in U_1$. 
On the other hand, if $\Theta_{u,t}(I_u) \cap \bigcup_{i=1}^{g(V)} E_i=\emptyset$, 
then $\Theta_{u,t}(I_u) \in U_2$. 
This means that 
$\Theta_{u,t}(I_u) \in U_1 \cup U_2$ for $u \in \partial D^2$ and $t \in [0,1)$.
Therefore, we conclude that $[\vph]=0 \in \pi_2(U,U_1 \cup U_2)$. 
\end{proof}

\subsection{Proof of Corollary~\ref{cor:conclusion}}
\label{sub:corollary}
Finally, we prove Corollary~\ref{cor:conclusion}. 

\vspace{0.5em}
\noindent 
\textit{Proof of Corollary~\ref{cor:conclusion}.} 
By Theorems~\ref{thm:amalgamation} and \ref{thm:handlebody}, 
$\Isot(M, \Sig)$ is finitely presented. 
It remains to show that $\MCG(M,\Sig)$ is finitely presented. 
By definition, there exists an exact sequence 
\[1 \rightarrow \Isot(M,\Sig) \rightarrow \MCG(M,\Sig) \rightarrow \MCG(M).\]
By Theorem~\ref{thm:dist-hyp}, $M$ is hyperbolic, and hence 
$\MCG(M)$ is finite. 
Therefore, $\MCG(M,\Sig)$ is finitely presented. 
\qed 
\vspace{0.5em}

\bibliographystyle{amsrefs}

\end{document}